\documentclass[12pt]{amsart}
\usepackage{amsmath,amssymb,amsfonts,amsthm,xspace}
\usepackage{graphicx}
\usepackage{geometry}
\usepackage[linktocpage=true]{hyperref}

\newtheorem{lemma}{Lemma}
\newtheorem{theorem}{Theorem}

\newcommand{\R}{\mathbb{R}}
\newcommand{\Z}{\mathbb{Z}}
\newcommand{\E}{\mathbb{E}}
\renewcommand{\P}{\mathbb{P}}
\newcommand{\old}[1]{}
\newcommand{\G}{\mathcal{G}}
\newcommand{\be}{\begin{equation}}
\newcommand{\ee}{\end{equation}}
\renewcommand{\H}{\mathcal{H}}
\newcommand{\HH}{\mathbb{H}}

\usepackage[usenames,dvipsnames]{color}

\begin{document}
\title{Random two-component spanning forests}

\author{Adrien Kassel}
\address{A. K.: D\'epartement de Math\'ematiques et Applications, \'Ecole Normale Sup\'erieure, 45 rue d'Ulm, 75230 Paris Cedex 05, France.}

\author{Richard Kenyon}
\address{R. K.: Mathematics Department, Brown University, 151 Thayer St., Providence, RI 02912.}

\author{Wei Wu}\address{W. W.: Applied Mathematics Department, Brown University, 182 George St., Providence, RI 02912.}

\thanks{A. K. was partially supported by Fondation Sciences Math\'ematiques de Paris. The research of R. K. 
was supported by the NSF}

\date{}

\begin{abstract} We study random two-component spanning forests ($2$SF) of finite graphs,
giving formulas for the first and second moments of the sizes of the components,
vertex-inclusion probabilities for one or two vertices, and the probability
that an edge separates the components. We compute the limit of these quantities when the graph tends to an infinite periodic graph in $\R^d$.
\end{abstract}

\maketitle

\section{Introduction}

For $\G$ a finite connected graph with vertex set $V$, a \emph{spanning tree} is a subgraph $(V,A)$, where $A$ is a set of edges, which contains no cycles and is connected. A \emph{two-component spanning forest} ($2$SF) is a subgraph $(V,B)$, where $B$ is a set of edges, which contains no cycles
and has exactly two connected components. A spanning tree of an $n$-vertex graph has $n-1$ edges;
a $2$SF has $n-2$ edges.

The matrix-tree theorem \cite{Kirch} (Theorem \ref{MTthm} below) equates the 
number of spanning trees $\kappa=\kappa(\G)$ with the determinant of the reduced Laplacian.
This result has led to an extensive study of the random spanning tree on many different families of graphs,
see e.g. \cite{BP, Wil, BLPS}.

Let $\kappa_2=\kappa_2(\G)$ be the number of $2$SFs. The ratio $\kappa_2(\G)/\kappa(\G)$ has an explicit
expression in terms of the potential kernel which follows from~\cite{Myr} as explained in~\cite{KW}. It reads
\begin{equation}  \label{ratio}
\frac{\kappa_2(\G)}{\kappa(\G)}=\sum_{uv\in
E}A_{u,v}A_{v,u}+\left(A_{u,v}-A_{v,u}\right)^2\,,
\end{equation}
where $A_{u,v}=G^{r}_{u,u}-G^{r}_{u,v}$ is the potential kernel and $G^r$ is the Green's function with Dirichlet boundary conditions at some vertex $r$. 

This implies the following theorem (which was previously obtained by other means in~\cite{KenWil,LP}).
\begin{theorem}\cite{KenWil, LP, KW}
For the $n\times n$ grid $\G$ we have
$$\kappa_2(\G) = \kappa(\G)\frac{n^2}{8}(1+o(1))\,,\,\text{as}\,\,n\to\infty\,.$$
\end{theorem}

In this paper, we compute various properties of the random $2$SF on the grid and other graphs. In particular we
give exact expressions for the first two moments of the volume
of the components, as well as vertex-inclusion probabilities. 

A first application concerns the Abelian sandpile model. By \cite{IKP}, waves of topplings of avalanches of sandpiles started at a vertex~$v$ are in bijection with $2$SFs where $v$ is disconnected from the sink. Theorem~\ref{main} below hence yields the first moment of the volume covered by each such wave (see Section~\ref{pinbush}), where the sink is the boundary vertex.

A second application, in the planar case, concerns cycle-rooted spanning trees. 
On a planar graph the dual of a $2$SF is a \emph{cycle-rooted spanning tree (CRST)}, that is, a set of $n=|V|$ edges connecting all vertices (and thus containing a unique cycle).  
By \cite{LP}, the expected length of the cycle in the scaling limit is related to the
so-called ``looping constant" of loop-erased random walk (LERW), the density of sand in recurrent Abelian sandpiles, and
derivatives of the Tutte polynomial at $( 1,1) $ (see also \cite{KenWil, KW}).  
Our results can be interpreted as computing the first two moments of the area of the unique cycle,
as well as the probabilities that the cycle encloses a given face or pair of faces, and the probability that
an edge is in the unique cycle.
 
\section{Finite graphs}
\label{2SF}

\subsection{Spanning trees and potential theory}
Let $\G=(V,E)$ be an undirected finite connected graph, 
endowed with a function $c:E\rightarrow \R_{>0}$ (which we call \emph{conductance} or \emph{weight}), 
and $b$ a marked vertex. 
The \emph{Laplacian operator} $\Delta:\R^V\to\R^V$ is defined by 
$$\Delta f(v) = \sum_{v'\sim v}c_{vv'}(f(v)-f(v'))$$
where the sum is over the neighbors of $v$.
The \emph{Dirichlet Laplacian} $\Delta_D$ is defined on~$\R^{V\setminus\{b\}}$
by the same formula (in which the sum, however, ranges over all of $V$, not just $V\setminus\{b\}$); 
in the natural basis indexed by $V$, $\Delta_D$ is 
the submatrix of $\Delta$ obtained by removing $b$'s row and column.

The operator $\Delta_D$ is invertible, see Theorem \ref{MTthm} below. 
Let $G$ denote the \emph{Green's function
with Dirichlet boundary conditions at $b$}; it is the inverse of $\Delta_D$. 

Entries of $G$ have both probabilistic and potential-theoretic interpretations:
$G_{u,v}$ is the expected number of visits to $v$ of a conductance-biased
random walk from $u$ to $b$. It is also the voltage at $v$ when one unit of current
flows from $u$ to $b$. $G_{u,u}$ in particular is the resistance between $u$ and $b$ \cite{DS}. See also
(\ref{Guuinterp}) below.

Given two directed edges $e=u_1v_1,e'=u_2v_2$ we define the \emph{transfer current} to be
$$T(e,e') = c(e')\left(G(u_1,u_2)-G(u_1,v_2)-G(u_2,v_1)+G(v_1,v_2)\right)\,.$$

The transfer current is used to compute edge inclusion probabilities for random spanning trees \cite{BP},
for example 
\be\label{treeedgeprob}\Pr(\text{edge $e$ is in the tree}) = T(e,e).\ee
The quantity $T(e,e')$ is also the amount of current crossing edge $e'$ when one unit of current
flows in at $u_1$ and out at $v_1$.


\old{
For a subgraph $\H$ of $\G$, the \emph{weight} of $\H$  
is defined by $w(\H)=\prod_{e\in\H}c(e)$. The \emph{weighted sum} of any collection $\HH$ of subgraphs 
of $\G$ is defined by $w({\HH})=\sum_{\H\in {\HH}}w(\H)$.
}
 
We define 
$$\kappa=\kappa (\G)= \sum_{\text{trees}~T}\,\prod_{e\in T}c(e)$$ 
to be the weighted sum of the collection of spanning
trees.

\begin{theorem}[The matrix-tree theorem \cite{Kirch}]\label{MTthm}
$\kappa(\G)=\det\Delta_D$.
\end{theorem}

\old{
A \emph{CRST} or \emph{unicycle} is a spanning subgraph which is connected and has a unique cycle,
that is, is the union of a spanning tree and an additional edge. 
We let $\lambda (\G)$ be the weighted sum of the
collection of CRSTs.

Note that on a planar graph the dual of a $2$SF (take duals of all edges not in the $2$SF) is a CRST. 
It is natural to assign conductances to the planar dual $\G^*$ which are the reciprocals of the conductances of $\G$;
then duality gives a bijection from $2$SFs to CRSTs which multiplies the weight by a constant (the reciprocal of the
product of all conductances of $\G$).
}

For general graphs a useful identity is 
\be\label{Guuinterp} G_{u,u}=\frac{\kappa(\G_{u\sim b})}{\kappa(\G)}\ee
where $\G_{u\sim b}$ is the graph $\G$ with $u$ and $b$ identified. This can be proved
by comparing the Laplacian of $\G$ and $\G_{u\sim b}$ which differ
in only a single entry; see e.g. \cite{Kirch}.

\subsection{Vertex-inclusion probabilities}
For any $2$SF of $\G$, we define the \emph{floating component} to
be the component not containing $b$.

In the following, let $\P$ denote the probability measure on $2$SFs which assigns to each $2$SF a probability proportional to its weight.
Let $\kappa_2=\kappa _{2}(\G)$ be the weighted sum of $2$SFs.

\begin{theorem}\label{3}
Let $\Sigma$ be the floating component of a $\P$-random $2$SF on $\G$. 
The probability that vertex $u$ is in $\Sigma$ is
\be\P(u) = \frac{\kappa}{\kappa _{2}}G_{u,u}.\label{1m}\ee
The probability that two vertices $u$ and $v$ are in $\Sigma $ is 
\be
\P(u,v)=\frac{\kappa}{\kappa _{2}}G_{u,v}.
\label{2m}
\end{equation}
\end{theorem}

\begin{proof}
Let $\G_{u}$ be the graph $\G$ with an additional edge $e_{u}$
connecting the wired boundary $b$ to $u$. The event $\{u\in \Sigma \}$ has
an interpretation in terms of spanning trees of~$\G_{u}$: it is the
event that $e_{u}$ is contained in a spanning tree of $\G_{u}$. Thus, letting $T_{\G_u}$ denote the transfer current on~$\G_u$, we have
\begin{equation}
\P (u)=\frac1{\kappa_2(\G)}\sum_{\stackrel{\mbox{\tiny\text{spanning trees $T$}}}{\mbox{\tiny\text{containing $e_{u}$}}}}w(T)=\frac{\kappa (\G_u)}{\kappa _{2}(\G)}T_{\G_u}(e_u,e_u)=
\frac{\kappa (\G)}{\kappa _{2}(\G)}G_{u,u}\,,
\end{equation}
where the second equality follows from~\eqref{treeedgeprob} and the third one from 
(\ref{Guuinterp}).

We now condition on the event that $u\in \Sigma $. Wire $u$ and $b$ together 
and construct a spanning tree of $\G_{u\sim b}$ using Wilson's algorithm (\cite{Wil}) starting at $v$. 
The conditional probability $\P(v\in\Sigma~|~u\in\Sigma)$, as a function of $v$, is the harmonic function with 
boundary values $1$ at $u$ and $0$ at $b$, hence we have 
\begin{equation}
\P(v\in\Sigma~|~u\in\Sigma)=\frac{G_{u,v}}{G_{u,u}}.  \label{condm}
\end{equation}
The result follows.
\end{proof}

\begin{theorem}\label{bdyedgethm} The probability that edge $e$ connects $\Sigma$ to $\Sigma^c$ is
\be\label{bdyedge}\P(e\in\partial\Sigma) = \frac{\kappa T(e,e)}{c(e)\kappa_2}.\ee
\end{theorem}

\begin{proof} By \cite{BP}, see \eqref{treeedgeprob}, $\kappa T(e,e)$ is the weighted sum of spanning trees containing edge $e$.
\end{proof}

\old{
By duality, for a CRST on the dual of a planar graph $\G$, the probability that edge~$e^*$, dual of edge~$e$, 
is in the unique cycle is
$$\P(e^*~\text{in cycle}) = \frac{\kappa(\G)}{c(e)\kappa_2(\G)}T_{\G}(e,e)=
\frac{c(e^*)\kappa(\G^*)}{\lambda(\G^*)}\left(1-T_{\G^*}(e^*,e^*)\right).$$}

\subsection{First and second moments of the size of $\Sigma$}

Let $\partial \Sigma $ denote the boundary of $\Sigma $, that is, the set of
edges with exactly one endpoint in $\Sigma $. Let $|\partial\Sigma|$ denote the sum of
weights of edges in $\partial\Sigma.$

\begin{lemma}
\be\label{lstar}
\mathbb{E}(|\partial\Sigma|) =
\frac{\sum_{e\in E}\kappa (\G)\P (e\in T)}{\kappa _{2}(\G)}  =\frac{\kappa (\G)(|V|-1)}{\kappa _{2}(\G)}.
\ee
\end{lemma}

\begin{proof}
The first equality follows from (\ref{bdyedge}), upon multiplying both sides of (\ref{bdyedge}) by $c(e)$ and
then summing over all edges.
The second equality follows from the fact that 
every spanning tree has exactly $|V|-1$ edges.
\end{proof}

Let $\ell^*=\E(|\partial\Sigma|)$ be the quantity in (\ref{lstar}). Summing
(\ref{1m}) and (\ref{2m}) over all vertices, we obtain the
following volume moments.

\begin{theorem}
\label{main} We have 
\begin{equation*}
\mathbb{E}(|\Sigma |)={\ell ^{\ast }}\frac{\left\vert V\right\vert }{
\left\vert V\right\vert -1}R\,\quad \text{and}\quad \,\mathbb{E}(|\Sigma
|^{2})={\ell ^{\ast }\frac{|V|}{\left\vert V\right\vert -1}\mathbb{E}
(\tau _{b}),}
\end{equation*}
where ${R=\sum_{v\in V}G_{v,v}/|V|}$ is the mean resistance between $v$ and $b$
for a uniform random $v$, and 
$$\mathbb{E}(\tau _{b})=\frac1{|V|}\sum_{u,v\in
V}G_{u,v}$$ is the expected hitting time to $b$ for the conductance-biased
random walk started at a uniform starting vertex.
\end{theorem}

\subsection{Pinned bush}\label{pinbush}

For any vertex $z_0\ne b$, let $\P_{z_0}$ be the probability distribution of the random $2$SF
conditioned so that its floating component $\Sigma$ contains~$z_{0}$. 

Using the expression of $\mathbb{P}\left( u\in \Sigma |z_{0}\in\Sigma\right) $ in~\eqref{condm},
and summing over $u$, we obtain that the conditional expected size of $\Sigma$ satisfies 
$$
\E\left(|\Sigma|~:~z_0\in\Sigma\right)=\sum_{v}\frac{G_{v,z_0}}{G_{z_0,z_0}}={\frac{\E(\tau^{z_0}_D)}{G_{z_0,z_0}}},$$
where 
$\tau^{z_0}_D$ is the exit time of a conductance-biased random walk started at $z_0$.

\section{Infinite graphs}

Let $\G=\Z^d$ with constant conductances $1$.
Let $\G_n=\G\cap[-n,n]^d$.
We wire all vertices of $\G\setminus\G_n$ into a single vertex which plays the role of the 
boundary of $\G_n$. 

In this setting the potential kernel on $\G_n$ converges to the potential kernel on $\G$.
On $\G_n$ we have $A_{u,v}=1/2d+o(1)$ for any edge $uv$ not within $O(1)$ of the boundary.
Since there are $dn^d(1+o(1))$ edges in $\G_n$, formula (\ref{lstar}) gives  
$$\kappa_2(\G_n) = \kappa(\G_n)n^d/(4d)(1+o(1)).$$

By (\ref{ratio}), the expected boundary size of the floating component is then
\be\label{4d}\E(|\partial\Sigma|)=4d+o(1).\ee

More generally, let $\G$ be a graph in $\R^d$, periodic under translations in $\Z^d$. 
Again let $\G_n=\G\cap[-n,n]^d$.
In this setting again the potential kernel on $\G_n$ converges to the potential kernel on $\G$.

Combining (\ref{ratio}) and (\ref{lstar}), we find (to leading order)
$$\E(|\partial\Sigma|) = n_0\left[\sum_{e\in \text{f.d.}} A_{u,v}A_{v,u}+\left(A_{u,v}-A_{v,u}\right)^2\right]^{-1},$$
where the sum is over edges in a single fundamental domain $[0,1)^d$, and $n_0$ is the number of vertices
per fundamental domain.

\section{Euclidean domains}

Here we consider scaling limits.
Although our results apply in greater generality (see the last but one paragraph of this section), for simplicity we consider only the case of subgraphs of $\Z^d$.

Let $D$ be a domain in $\mathbb{R}^{d}$ (for some $d\geq 2$) with boundary a smooth hypersurface.
Let $\G_{n}$ be the nearest-neighbor graph of $\frac1n\Z^d$ with all vertices outside of $D$ wired to a single vertex called $b$.  All edges have conductance $1$.

By (\ref{4d}) above, $\E_n(|\partial\Sigma|)\to 4d$ as $n\to\infty$.
For general graphs $\G$ we denote this limit $\ell^*=\ell^*(\G)$ when it exists.

Denote $R_{n}$ to be the mean resistance (from a uniformly chosen vertex to the wired
boundary) associated with $\G_{n}$. We obtain the following.

\begin{theorem}
\label{maininfinite} Suppose $d\geq 2$. Let $z\ne z'\in D$, and $z_{n},z_{n}'$ be
points on $\G_n$ within distance $O\left( 1/n\right) $ of $
z,z'$, respectively. As $n\rightarrow \infty $, we have 
\begin{equation*}
\P_n (z_{n},z_{n}'\in\Sigma)=\frac{4d}{|D|n^{2d-2}}g_D^0(z,z')\,+o( n^{-d}),
\end{equation*}

As $n\rightarrow \infty $, we have (for $d\geq 3$)
\begin{eqnarray*}
\E_n(|\Sigma|)&=&4dR^{\ast }+o(1)\\
\E_n(|\Sigma|^{2})&=&4dC(D)|D|n^2+o(n^2),
\end{eqnarray*}
where $R^{\ast }=\lim_{n\rightarrow \infty }R_{n}$, and $C(D)=\lim_{n\to\infty}\frac{\E(\tau_b)}{|D|n^2}$ is the
expected exit time from $D$ for Brownian motion started at a uniform point in $D$, divided by $|D|$\old{(to make it scale invariant)}.

For $d=2$, the expression for $\mathbb{E}_n(|\Sigma|^{2})$ is
the same above, and we have \old{
$$
\P (z_{n},z_{n}')=\frac{4}{\pi n^2}\log
|z-z'|\quad +o( n^{-2})$$
}
$$
\E_n(|
\Sigma|)=\frac{4\log n}{\pi}+o(\log n).
$$
\end{theorem}

\begin{proof}
The sequence of graphs $( \G_{n}) _{n\geq 1}$
has the following approximation property: the discrete Green's function
converges under rescaling by $n^{2-d}$ to the continuous Green's function $g_{D}^{0}$ on $D$ with
Dirichlet boundary condition \cite{LL}: $n^{d-2}G(z_n,z_n')\to g^0_D(z,z')$.

The result follows by passing to the limit in Theorem~\ref{main} and using convergence of random walk to Brownian motion to get the convergence of the expected exit time (with the right Brownian time-space scaling). The convergence of the mean resistance to a limit on the infinite network follows from Rayleigh's principle~\cite{BLPS}.

In particular, $R^{\ast }<\infty $ exists for $d\geq 3$ by transience of the
random walk \cite{DS}. For $d=2$, $\lim_{n\rightarrow \infty }R_{n}/\log
n={1}/({2\pi})$ by explicit asymptotics of the Green's kernel~\cite{LL}.

\old{More simply, the result may also be obtained by applying~\eqref{1m}, \eqref{2m}, Theorem~\ref{main}, and using the convergence of Riemann sums for the discrete Green's function to integrals of the continuous Green's function. }
\end{proof}

\old{As already pointed out, for $\Z^d$, $\ell^*=4d$. }
The mean resistance is the normalized trace of the Green's function, and so for the cubic grid of sidelength $n$
can be computed by explicit diagonalization. We have
\be
R_n=n^{-d}\sum_{k_1,\ldots, k_d\stackrel{*}{=}1}^n\left(4\sum_{i=1}^d\sin^2(\pi k_i/n)\right)^{-1},
\ee
where the $\stackrel{*}{=}$ indicates that we leave off the term in which all $k_j=n$,
and
$$R^*=\frac1{(2\pi)^d}\int_{[0,2\pi]^d}\frac1{2d-2\cos\theta_1-\dots-2\cos\theta_d} d\theta_1\dots d\theta_d.$$

For fixed $|D|$, the constant $C(D)$ is maximized for a sphere~\cite{Po}. In the case of a
cuboid $D$ in $\mathbb{Z}^{d}$ of side lengths $a_{1},\ldots ,a_{d}$, it is equal 
(the proof uses the explicit expression for the eigenvalues of the Laplacian) to

\begin{equation}
C(D)=\frac{4^{2d}}{\pi^{2d+2}}\sum_{n_{i}\ge 1,{\text{odd }}}
\prod_{i=1}^{d}\frac{1}{a_in_i^2}\left(\sum_{i=1}^d\frac{n_{i}^{2}}{a_i^2}\right)^{-1}
\,.  \label{drect}
\end{equation}

It is interesting to note that in dimension two, $C(D)= P(D)/|D|^2$, where $P(D)$ is what P\'olya calls the \emph{torsional rigidity} of the cross-section $D$, and which is, in mechanical terms, a measure of the resistance to torsion of a cylindrical beam with cross-section $D$, defined by $1/P(D)=\inf_{f}w(f)$ is the infimum, over all smooth functions $f$ over $D$ vanishing on the boundary, of $w(f):=\frac{\int_{D}\left|\nabla f\right|^2}{4\left(\int_{D}f\right)^2}$.

Theorem~\ref{maininfinite} is valid in greater generality: one can take any periodic graph in $\R^d$ for which random walk converges to Brownian motion, and replace the constants~$4d$ in the statement with $\ell^\ast$. Indeed, we need only the property that the discrete Green's function converges to
the continuous one. 
Both $\ell^*$ and $R^*$ can be computed, since the Green's function is an explicit integral of a rational function
and $\ell^*$ and $R^*$ are obtained from the Green's function.
For the square, triangular, and hexagonal lattices $\ell^*$ is easily computed;
we list the relevant quantities for these cases in Table~\ref{constants}.

\begin{table}[tb]
\centering
\begin{tabular}{|l|l|l|l|}
\hline
& square grid & hexagonal grid & triangular grid \\ 
\hline
\hline
$\ell^*$ & $8$ & $6$ & $12$ \\ 
\hline
$\mathbb{E}_n(A)/(\log n/\pi)$ & $4$ & $3\sqrt{3}$ & $2 \sqrt{3}$ \\ 
\hline
$\mathbb{E}_n(A^2)/ (C(D) n^2)$ & $4$ & $3\sqrt{3}$ & $2 \sqrt{3}$\\
\hline
\end{tabular}%
\caption{Constants for three planar regular lattices}\label{constants}
\label{table}
\end{table}

\section{Spanning unicycles on planar graphs}

A \emph{cycle-rooted spanning tree} (CRST), or \emph{unicycle}, is a spanning subgraph which is connected and has a unique cycle,
that is, is the union of a spanning tree and an additional edge. 
We let $\lambda (\G)$ be the weighted sum of the
collection of CRSTs.

On a planar graph the dual of a $2$SF (take duals of all edges not in the $2$SF) is a CRST. 
It is natural to assign conductances to the planar dual $\G^*$ which are the reciprocals of the conductances of $\G$;
then duality gives a bijection from $2$SFs to CRSTs which multiplies the weight by a constant (the reciprocal of the
product of all conductances of $\G$).

In the planar case, we can use planar duality to translate the previous
statements about the floating component of a $2$SF into statements about
the unique loop of a weighted spanning unicycle.

On a planar graph embedded in the plane, the dual of a spanning unicycle on $%
\G$ is a $2$SF on the dual graph $\G^*$, for which we choose
the marked vertex $b^*$ to be the outer boundary face of $\G$. Let $V^*$
be the vertex set of $\G^*$.

Theorem \ref{bdyedgethm} in this setting shows
that for a CRST on the dual of a planar graph $\G$, the probability that edge~$e^*$, dual of edge~$e$, 
is in the unique cycle is
$$\P(e^*~\text{in cycle}) = \frac{\kappa(\G)}{c(e)\kappa_2(\G)}T_{\G}(e,e)=
\frac{c(e^*)\kappa(\G^*)}{\lambda(\G^*)}\left(1-T_{\G^*}(e^*,e^*)\right).$$

Theorem \ref{3} and Theorem \ref{main} translate to the following
result in this setting. Denote~$A$ to be the area (i.e. the number of
faces enclosed) of the unique cycle of a random spanning unicycle.

\begin{theorem}
Let $f,f'$ be two faces of $\G$. Then 
\begin{equation*}
\P(f \text{ enclosed}) = \frac{\kappa(\G)}{\lambda(\G)}G^*_{f,f}\qquad \P (f,f'\text{ enclosed})=
\frac{\kappa (\G)}{\lambda (\G)}
G_{f,f^{\prime }}^{\ast }.
\end{equation*}
We also have 
\begin{equation*}
\mathbb{E}(A)=\frac{\kappa (\G)}{\lambda (\G)}%
\sum_{v^{\ast }\in V^{\ast }}G_{v^{\ast },v^{\ast }}^{\ast }\,\quad \text{and%
}\quad \,\mathbb{E}(A^{2})=\frac{\kappa (\G)}{\lambda (\G)}%
|V^{\ast }|^{2}\mathbb{E}(\tau _{b^{\ast }})\,,
\end{equation*}%
where $\mathbb{E}(\tau _{b^{\ast }})$ is the expected hitting time of $b^{\ast }$ for a conductance-biased random walk started at a uniformly chosen
starting vertex of $\G^{\ast }$.
\end{theorem}

Let $\G_{n}$ be the $n\times n$ grid with unit conductances. 
In \cite{KK} were computed up to constants the moments of the combinatorial area $A$ of
the uniform unicycle on $\G_n$ (whose probability distribution we denote by $\nu_n$), or equivalently, the moments of the size of the floating component of the uniform $2$SF:

\begin{theorem}\cite{KK}\label{highermoms}
\label{moments} For all integer $k\ge 2$, there is a constant $C_k>0$, such that $\E\left(A_{n}^{k}\right)=C_{k}n^{2k-2}(1+o(1))$, as $n\to\infty$.
\end{theorem}

We give a sketch of the proof for completeness.

\begin{proof}
Let $\H_n$ be the graph $\G_n$ scaled to fit in the square $D=[0,1]^{2}$. Let $\mu_n$ be 
the measure on unicycles on $\H_n$, weighted by the square of the area of the cycle.
In \cite{KK} it is shown that $\mu_n$ converges as $n\to\infty$ to a measure $\mu$ with the property that
the probability of a cycle of positive area is positive. 
For a cycle of area $A$ in~$\G_n$, the Radon-Nikodym derivative between $\mu_n$ and $\nu_n$ is 
$d\mu_n/d\nu_n = A^{2}/\E_{\nu_n}(A^2)$. 

We have
$$\frac{\E_{\nu_n}(A^k)}{\E_{\nu_n}(A^2)} = \E_{\mu_n}(A^{k-2}) = n^{2k-4}\E_{\mu}(\theta^{k-2})(1+o(1)),$$
where $\theta$ is the scaled Euclidean area (in $[0,1]$). By Theorem \ref{maininfinite}, 
$$\E_{\nu_n}(A^2)=C(D)|D|n^2(1+o(1))$$
thus we have
$$\E_{\nu_n}(A^k)=n^{2k-2}C(D)\E_{\mu}(\theta^{k-2})(1+o(1)).$$
\end{proof}

\section{Questions}

\begin{enumerate}\item Can one compute the constants $C_k$ in the higher moments of Theorem \ref{highermoms}?

\old{
\item Can one compute the expected length of the cycle of the spanning unicycle
(or the looping constant of the LERW) on other periodic planar lattices? Can one compute $\ell^*$ in higher dimension, for example in $\Z^3$?}

\item Can one compute the expected length of the cycle of the spanning unicycle in higher dimension, for example in $\Z^3$?

\item The probability that three distinct vertices are in $\Sigma$ seems to be a much harder quantity to compute.
Can this probability be written in terms of the Green's function?
\end{enumerate}

\section*{Acknowledgements}

We thank C\'edric Boutillier, Yuval Peres, and David Wilson
for helpful discussions and feedback, as well as the anonymous referee for spotting an incorrection in the first version.

\end{document}